\documentclass[11pt, a4paper]{article}
\usepackage{amsmath, amsthm, amssymb, url}
\usepackage[margin=1in]{geometry}
\usepackage[english]{babel}
\usepackage{multirow}
\usepackage{authblk}
\newcommand{\End}{\mathrm{End}}

\newcommand{\id}{\mathrm{id}}

\newcommand{\CA}{\mathrm{CA}}
\newcommand{\ICA}{\mathrm{ICA}}

\newcommand{\Rank}{\mathrm{Rank}}

\newcommand{\Aut}{\mathrm{Aut}}

\theoremstyle{plain}

\newtheorem{corollary}{Corollary}
\newtheorem{lemma}{Lemma}
\newtheorem{proposition}{Proposition}
\newtheorem{theorem}{Theorem}

\theoremstyle{definition}

\newtheorem{example}{Example}
\newtheorem{remark}{Remark}

\begin{document}

\title{Generating infinite monoids of cellular automata}
\author{Alonso Castillo-Ramirez\footnote{Email: alonso.castillor@academicos.udg.mx},  \\
\small{Department of Mathematics, University Centre of Exact Sciences and Engineering,\\ University of Guadalajara, Guadalajara, Mexico.} }

\maketitle

\begin{abstract}
For a group $G$ and a set $A$, let $\End(A^G)$ be the monoid of all cellular automata over $A^G$, and let $\Aut(A^G)$ be its group of units. By establishing a characterisation of surjunctuve groups in terms of the monoid $\End(A^G)$, we prove that the rank of $\End(A^G)$ (i.e. the smallest cardinality of a generating set) is equal to the rank of $\Aut(A^G)$ plus the relative rank of $\Aut(A^G)$ in $\End(A^G)$, and that the latter is infinite when $G$ has an infinite decreasing chain of normal subgroups of finite index, condition which is satisfied, for example, for any infinite residually finite group. Moreover, when $A=V$ is a vector space over a field $\mathbb{F}$, we study the monoid $\End_{\mathbb{F}}(V^G)$ of all linear cellular automata over $V^G$ and its group of units $\Aut_{\mathbb{F}}(V^G)$. We show that if $G$ is an indicable group and $V$ is finite-dimensional, then $\End_{\mathbb{F}}(V^G)$ is not finitely generated; however, for any finitely generated indicable group $G$, the group $\Aut_{\mathbb{F}}(\mathbb{F}^G)$ is finitely generated if and only if $\mathbb{F}$ is finite. \\

\textbf{Keywords:} cellular automata; linear cellular automata; rank of a monoid; finitely generated monoid; surjunctive group.     
\end{abstract}

\section{Introduction}\label{intro}

Let $G$ be a group and $A$ a set. The \emph{full shift} $A^G$ is the set of all maps $x : G \to A$ equipped with the \emph{shift action} of $G$ on $A^G$ defined by $(g \cdot x) (h) := x(g^{-1}h)$ for every $g,h \in G$, $x \in A^G$. We endow $A^G$ with the \emph{prodiscrete topology}, which is the product topology of the discrete topology on $A$. A neighborhood base at $x \in A^G$ is given by 
 \[  V(x, S) = \{ y \in A^G :  y \vert_S =  x \vert_S \},  \]
 where $S$ runs over all finite subsets of $G$.   

A \emph{cellular automaton} over $A^G$ is a function $\tau : A^G \to A^G$ such that there is a finite subset $S \subseteq G$, called a \emph{memory set} of $\tau$, and a \emph{local function} $\mu : A^S \to A$ satisfying
\[ \tau(x)(g) = \mu (( g^{-1} \cdot x) \vert_{S}),  \quad \forall x \in A^G, g \in G.  \]
Cellular automata were invented by John von Neumann and Stanislaw Ulam in the decade of 1940's, and popularised by John Conway's \emph{Game of Life} in 1970. Nowadays, cellular automata have become fundamental objects in several areas of mathematics, such as symbolic dynamics, complexity theory, and complex systems modeling, and its theory has flourished due to its diverse connections with group theory, topology, and dynamics (e.g. see the highly-cited book \cite{CSC10} and references therein). 

When $A$ is finite, cellular automata are characterised by Curtis-Hedlund-Lyndon theorem (see \cite[Theorem 1.8.1]{CSC10}), which establishes that a function $\tau : A^G \to A^G$ is a cellular automaton if and only if it is $G$-equivariant (i.e. $\tau(g \cdot x) = g \cdot \tau(x)$, for all $x \in A^G$, $g \in G$) and continuous in the prodiscrete topology of $A^G$. 

For any group $G$ and set $A$, define
\[ \End(A^G) := \{ \tau : A^G \to A^G \; \vert \;  \tau \text{ is a cellular automaton} \}. \]
As the composition of any two cellular automata is a cellular automaton, the set $\End(A^G)$, equipped with the composition of functions, is a monoid (see \cite[Corollary 1.4.11]{CSC10}); its \emph{group of units} (i.e., the group of invertible elements) is denoted by $\Aut(A^G)$. Sometimes $\End(A^G)$ and $\Aut(A^G)$ are also denoted by $\CA(G;A)$ and $\ICA(G;A)$, respectively. 

When $G = \mathbb{Z}$, the group $\Aut(A^{\mathbb{Z}})$ is an important object of study in symbolic dynamics (e.g. see \cite{BLR88,R72,S18,S19}). Several interesting properties are known for $\Aut(A^{\mathbb{Z}})$: it is a countable group that is not finitely generated and it contains an isomorphic copy of every finite group, as well as the free group on a countable number of generators (see \cite{BLR88} and \cite[Sec. 13.2]{LM95}). Furthermore, Ryan's theorem establishes that the centre of $\Aut(A^{\mathbb{Z}})$ is isomorphic to $\mathbb{Z}$. Many of the previous results may be generalised when replacing $A^\mathbb{Z}$ by a susbet $S$ which is a \emph{mixing subshift of finite type}. However, most of the algebraic properties of $\Aut(A^{\mathbb{Z}})$ still remain unknown; for example, several of the important open problems in symbolic dynamics chosen by Mike Boyle revolve around these groups (\cite[Sec. 22]{B08}). 

It is a natural question to ask which of the algebraic properties described in the previous paragraph still hold when we replace $\mathbb{Z}$ by an arbitrary group $G$. A first important step would be to prove that the group $\Aut(A^G)$ is not finitely generated, which is also a standard question in geometric group theory. In \cite[Theorem 7]{C20}, it was shown that the monoid $\End(A^G)$ is not finitely generated when $G$ has an infinite descending chain of normal subgroups of finite index, and that large families of infinite groups satisfy this condition, such as infinite residually finite and infinite locally graded groups. No analogous result is yet known for $\Aut(A^G)$, except for very specific cases such as $G=\mathbb{Z}$. 

An important tool that was used in the proof of \cite[Theorem 7]{C20} is the \emph{rank of a monoid} (i.e., the smallest cardinality of a generating set of the monoid). This has been studied for several kinds of finite and infinite monoids, and, in particular, it has been studied for monoids of cellular automata over finite groups in \cite{CRG19,CRG16a,CRSA19}. 

In the first part if this paper, by providing a new characterisation of surjunctive groups, we show that, if $G$ is surjunctive, then, for every finite set $A$,
\begin{equation}\label{eq-rank-formula}
\Rank(\End(A^G)) = \Rank(\Aut(A^G)) + \Rank(\End(A^G) : \Aut(A^G)), 
\end{equation}
where $\Rank(\End(A^G) : \Aut(A^G))$ is the relative rank of $\Aut(A^G)$ in $\End(A^G)$, i.e. the smallest cardinality of a subset $W \subseteq \End(A^G)$ such that $W \cup \Aut(A^G)$ generates $\End(A^G)$. Furthermore, analogously to \cite[Theorem 7]{C20}, we establish that the relative rank of $\Aut(A^G)$ in $\End(A^G)$ is infinite whenever $2 \leq \vert A \vert < \infty$ and $G$ has an infinite descending chain of normal subgroups of finite index.

In the second part of this paper, we specialise on linear cellular automata when $A=V$ is a vector space over a field $\mathbb{F}$. In particular, we show that if $\mathbb{F}$ is infinite, or $V$ is infinite-dimensional, or $G$ is an indicable group (i.e. there exits an epimorphism from $G$ to $\mathbb{Z}$), then the monoid $\End_{\mathbb{F}}(V^G)$ of all linear cellular automata over $V^G$ is not finitely generated. However, we show that if $G$ is finitely generated indicable, then the group $\Aut_{\mathbb{F}}(\mathbb{F}^G)$ of all invertible linear cellular automata over $\mathbb{F}^G$ is finitely generated if and only if $\mathbb{F}$ is a finite field. These results settle the question of the finite generation for important classes of groups and monoids of linear cellular automata.

%%%%%%%%%%%%%%%%%%%%%%%%%%%%%%%%%%%%%%%%%%%%%%%%%%%

\section{Ranks of monoids} \label{sec:monoids}

In this section we review some basic properties of ranks of monoids that will be essential in the sequel. Although these results are not new, we have not found a self-contained source that includes all of them, so we believe it important to reproduce them here. 

Let $M$ be a monoid with identity $1$. An element $u \in M$ is a \emph{left unit} of $M$ if there exists $a \in M$ such that $au = 1$, and $a$ is called a \emph{left inverse} of $u$. Similarly, $u \in M$ is a \emph{right unit} of $M$ if there exists $b \in M$ such that $ub = 1$, and $b$ is called a \emph{right inverse} of $u$. The element $u \in M$ is simply called a \emph{unit} of $M$ if it is both a left and right unit of $M$.

\begin{lemma}\label{left-right}
If $u \in M$ has a left inverse $a \in M$ and a right inverse $b \in M$, then $a=b$. 
\end{lemma}
\begin{proof}
By associativity, $(au)b = a(bu)$. As $(au)b = 1b = b$ and $a(bu) = a1 = a$, we conclude $a=b$.
\end{proof}

Let $U_L(M)$, $U_R(M)$ and $U(M)$ the sets of left, right and two-sided units of $M$, respectively. It is easy to verify that $U_L(M)$ and $U_R(M)$ are submonoids of $M$, and $U(M) = U_L(M) \cap U_R(M)$ is a subgroup of $M$.

An \emph{ideal} of $M$ is a subset $I \subseteq M$ such that $ra \in I$ and $ar \in I$ for all $r \in I$, $a \in M$. The monoid $M$ is called \emph{directly finite} if, for every $a,b \in M$, $ab = 1$ implies $ba=1$.

The proof of following result is routine (c.f. \cite[Lemma 2.3]{E2019}). 

\begin{lemma}\label{le-monoids}
Let $U_L = U_L(M)$, $U_R = U_R(M)$ and $U = U(M)$. The following are equivalent:
\begin{enumerate}
\item $M$ is directly finite.
\item $U_L = U$.
\item $U_R= U$.
\item $M \setminus U$ is an ideal of $M$. 
\end{enumerate}
\end{lemma}
%\begin{proof}
%That (1.) implies (2.) follows directly from the definition of directly finite monoid. 
%
%Suppose that (2.) holds, and let $u \in U_R$. There exists $a \in M$ such that $ua = 1$, which implies that $a \in U_L = U$. Hence, $a$ must have a right inverse, which by Lemma \ref{left-right} must be $u$. Thus, $ua=1$, which shows that $u \in U$, and $U_R=U$.
%
%Suppose that (3.) holds. Let $r \in M \setminus U$ and $m \in M$. Assume that $rm \in U$. Then, there exists $a \in M$ such that $1 = (rm)a= r(ma)$, which implies that $r \in U_R=U$, a contradiction. Hence, $rm \in M \setminus U$. Now, assume that $mr \in U$. Then, there exists $a \in M$ such that $1 = (mr)a = m(ra)$, which implies that $m \in U_R = U$. By Lemma \ref{left-right}, $1 = (ra)m = r(am)$, which implies that $r \in U_R = U$, so we have again a contradiction. Hence, $mr \in M \setminus U$, and $M \setminus U$ is an ideal of $M$.
%
%Suppose that (4.) holds, and assume that $ab=1$ for $a, b \in M$. If $b \in M \setminus U$, then $1 = ab \in M \setminus U$, which is a contradiction. Therefore, $b \in U$, which imples that $ba = 1$ by Lemma \ref{left-right}. 
%\end{proof}

\begin{example}
Any finite monoid $M$ is directly finite. To see this, we may use the generalized Cayley's theorem to identify $M$ with a submonoid of a full transformation monoid, and use the fact that for self-maps of a finite set being injective, surjective, right unit, and left unit are all equivalent. 
\end{example}

Given a subset $T$ of $M$, the \emph{submonoid generated} by $T$, denoted by $\langle T \rangle$, is the smallest submonoid of $M$ that contains $T$; this is equivalent as defining $\langle T \rangle := \{ t_1 t_2 \dots t_k \in M : t_i \in T, \ \forall i, \ k \geq 0 \}$. We say that $T$ is a \emph{generating set of $M$} if $M= \langle T \rangle$. The monoid $M$ is said to be \emph{finitely generated} if it has a finite generating set. The \emph{rank} of $M$ is the minimal cardinality of a generating set:
\[ \Rank(M) := \min\{\vert T \vert : M= \langle T \rangle \}. \]	

The $\Rank$ function on monoids does not behave well when taking submonoids or subgroups: in other words, if $N$ is a submonoid of $M$, there may be no relation between $\Rank(N)$ and $\Rank(M)$. It is even possible that $M$ is finitely generated but $N$ is not finitely generated (such is the case of the free group on two symbols and its commutator subgroup).

For any subset $R$ of a monoid $M$, the \emph{relative rank} of $R$ in $M$ is
\[ \Rank(M:R) = \min \{ \vert W \vert : M = \langle R \cup W \rangle  \}. \]

The next result was shown in \cite[Proposition 4]{E2012} for factorizable inverse monoids, but its proof may be identically adapted for any monoid in which $M \setminus U(M)$ is an ideal: the key fact used in the proof is that if $T$ is a generating set for $M$, then $T$ must contain a generating set for the group of units of $M$.    

\begin{lemma}\label{rank-formula}
Let $M$ be a monoid and $U$ its group of units. If $M$ is directly finite, then
\[ \Rank(M) = \Rank(U) + \Rank(M:U). \]
\end{lemma}

The following result is a useful tool to find lower bounds on ranks and relative ranks. 

\begin{lemma}\label{le-epi-mon}
Let $\phi : M_1 \to M_2$ be a epimorphism of monoids. Then:
\begin{enumerate}
\item $\Rank(M_2) \leq \Rank(M_1)$.
\item $\Rank(M_2 : U_2) \leq \Rank(M_1 : U_1)$, where $U_i$ is the group of units of $M_i$.
\end{enumerate}
\end{lemma}
\begin{proof}
For part (1.), let $T$ be a generating set for $M_1$ of smallest cardinality. Then $\phi(T)$ is a generating set for $M_2$, and
\[ \Rank(M_1) = \vert T \vert \geq \vert \phi(T) \vert  \geq \Rank(M_2). \]
For part (2.), let $W$ be a set of smallest cardinality such that $\langle U_1 \cup W \rangle = M_1$. We show that $U_2 \cup \phi(W)$ generates $M_2$. For every $m_2 \in M_2$ there exists $m_1 \in M_1$ such that $\phi(m_1) = m_2$. Then, there exist $a_1, a_2, \dots, a_t \in U_1 \cup W$ such that $m_1 =a_1a_2 \dots a_t$. Hence,
\[ m_2 = \phi(m_1) =  \phi(a_1) \phi(a_2) \dots \phi(a_t), \]
where $\phi(a_1), \phi(a_2), \dots, \phi(a_t) \in \phi(U_1) \cup \phi(W) \subseteq  U_2 \cup \phi(W)$, since $\phi(U_1) \subseteq U_2$. Therefore,
\[ \Rank(M_1 : U_1) = \vert W \vert \geq \vert \phi(W) \vert \geq \Rank(M_2 : U_2).   \]
\end{proof}

%%%%%%%%%%%%%%%%%%%%%
\section{Generating monoids of cellular automata} \label{CA}

In this section we study the ranks of $\End(A^G)$ and $\Aut(A^G)$; in particular, we want to decide when these ranks are finite or infinite.

\begin{lemma}\label{le-1}
Let $G$ be a group and $A$ a set. If $\End(A^G)$ is finitely generated, then $A$ is finite and $G$ is finitely generated. 
\end{lemma}
\begin{proof}
First note that, if $A$ is an infinite set, then $\End(A^G)$ is uncountable, as there are uncountably many local functions $\mu : A^S \to A$, for any finite subset $S \subseteq G$. It follows that $\Rank(\End(A^G)) = \vert \End(A^G) \vert$ (see \cite[p. 268]{E2012}). If $G$ is not finitely generated, then $\End(A^G)$ is not finitely generated by \cite[Remark 1]{C20}. 
\end{proof}

If $A$ is finite with less than two elements, then $\End(A^G)$ is a trivial monoid; hence, we shall assume that $A$ has at least two elements. If both $A$ and $G$ are finite, then $\End(A^G)$ is finite, so its rank must be finite. Finally, if $A$ is finite with at least two elements and $G$ is infinite finitely generated, then both $\End(A^G)$ and $\Aut(A^G)$ are countably infinite (as the set of local functions $\mu : A^S \to A$, where $S$ runs among all finite subsets of $G$, is countably infinite). Hence, their ranks might be finite or countably infinite in this last case. 

Our goal now is to give a necessary and sufficient condition for the monoid $\End(A^G)$ to be directly finite, as this will allow us to apply Lemma \ref{rank-formula}.

\begin{lemma}\label{le-prop}
Let $G$ be a group and $A$ a finite set. Let $\tau : A^G \to A^G$ be a cellular automaton. 
\begin{enumerate}
\item If $\tau$ is a right unit of $\End(A^G)$, then $\tau$ is surjective. 
\item $\tau$ is a left unit of $\End(A^G)$ if and only if $\tau$ is injective.
\item $\tau$ is a unit of $\End(A^G)$ if and only if $\tau$ is bijective.
\end{enumerate}
\end{lemma}
\begin{proof}
\begin{enumerate}
\item If $\tau$ is a right unit, there exists $\sigma \in \End(A^G)$ such that $\tau \sigma = 1$. Now, clearly $\sigma(x)$ is a preimage of any $x \in A^G$ under $\tau$, which means that $\tau$ is surjective. 
 
\item If $\tau$ is a left unit, there exists $\sigma \in \End(A^G)$ such that $\sigma \tau = 1$. If $\tau(x) = \tau(y)$, we apply $\sigma$ on both sides to conclude that $x=y$, so $\tau$ is injective. 

For the converse, suppose that $\tau$ is injective. Define a function $\phi : \tau(A^G) \to A^G$ as follows: for every $ x \in \tau(A^G)$, let $\phi(x)$ be the unique preimage of $x$ under $\tau$. Hence, $\tau(\phi(x)) = x$, for all $x \in \tau(A^G)$, and $\phi(\tau(z)) = z$, for all $z \in A^G$. We shall apply \cite[Lemma 7.8.2]{CSC10}, which establishes that if $X \subseteq A^G$ is a closed $G$-invariant set (i.e. for all $g \in G$, $x \in X$, we have $g \cdot x \in X$) and $f : X \to A^G$ is a continuous $G$-equivariant map, then there exists $\sigma \in \End(A^G)$ such that $\sigma \vert_X = f$. 

By \cite[Lemma 3.3.2.]{CSC10}, the set $\tau(A^G)$ is closed in $A^G$, and, using the $G$-equivariance of $\tau$ it is easy to check that $\tau(A^G)$ is a $G$-equivariant set.

Let $g \in G$ and $x \in \tau(A^G)$. Then $g \cdot \phi(x)$ is the preimage of $g \cdot x$ under $\tau$, because $\tau(g \cdot \phi(x)) = g \cdot \tau(\phi(x)) = g \cdot x$. This means that $\phi(g \cdot x) = g \cdot \phi(x)$, which shows that $\phi$ is $G$-equivariant. To show continuity, observe that for any finite $S \subseteq G$ and $x \in A^G$,
\begin{align*}
\phi^{-1}(V(x,S)) &= \{ y \in \tau(A^G) : \phi(y) \vert_S = x \vert_S  \}\\
& = \{ y \in \tau(A^G) : y \vert_S = \tau(x) \vert_S  \} \\
& = \tau(A^G) \cap V( \tau(x), S), 
\end{align*}
which is an open set in $\tau(A^G)$. Therefore, $\phi$ is continuous. Using \cite[Lemma 7.8.2]{CSC10}, we find $\sigma \in \End(A^G)$ such that $\sigma \vert_{\tau(A^G)} = \phi$. Therefore, $\sigma \tau (z) = z$, for all $z \in A^G$, which proves that $\tau$ is a left unit. 

\item See \cite[Theorem 1.10.2.]{CSC10}.
\end{enumerate}
\end{proof}

A group $G$ is \emph{surjunctive} if for every finite set $A$, every injective cellular automaton $\tau : A^G \to A^G$ is surjective (and hence bijective). Many groups are known to be surjunctive, including the families of locally finite, residually finite, abelian, free, amenable, and sofic groups (see Chapters 3 and 7 in \cite{CSC10}). In fact, it is an open question whether there exists a group that is not surjunctive. 

\begin{theorem}\label{th-surjunctive}
A group $G$ is surjunctive if and only if, for every finite $A$, the monoid $\End(A^G)$ is directly finite.  
\end{theorem}
\begin{proof}
Suppose that $G$ is surjunctive and let $\tau \in \End(A^G)$ be a left unit. By Lemma \ref{le-prop} (2.), $\tau$ is injective, which implies that $\tau$ is bijective, as $G$ is surjunctive. By Lemma \ref{le-prop} (3.), $\tau$ is a unit. It follows that $\End(A^G)$ is directly finite by Lemma \ref{le-monoids}.

Conversely, suppose that $\End(A^G)$ is directly finite. Let $\tau \in \End(A^G)$ be injective. By Lemma \ref{le-prop} (2.), $\tau$ is a left unit, and by Lemma \ref{le-monoids}, $\tau$ is a unit. Therefore, $\tau$ is bijective by Lemma \ref{le-prop} (3.).
\end{proof}

\begin{remark}
The converse of Lemma \ref{le-prop} (1.) is false. For example, it is known (see \cite[Example 3.8.8.]{CSC10}) that the cellular automaton $\tau \in \End(\{ 0,1 \}^\mathbb{Z})$ defined by $\tau(x)(i) = x(i+1) + x(i) \mod(2)$, for all $x \in \{ 0,1 \}^{\mathbb{Z}}$, $i \in \mathbb{Z}$, is surjective but not injective. If $\tau$ is a right unit, by the previous theorem it is a unit, as $\mathbb{Z}$ is surjunctive, and therefore it is bijective, which is a contradiction.    
\end{remark}

\begin{lemma}
Let $G$ be a surjunctive group and $A$ a finite set. Then,
\[ \Rank(\End(A^G)) = \Rank(\Aut(A^G)) + \Rank(\End(A^G) : \Aut(A^G)). \]
\end{lemma}
\begin{proof}
The result follows by Theorem \ref{th-surjunctive} and Lemma \ref{rank-formula}.
\end{proof}
\begin{corollary}
Let $G$ be a surjunctive group and $A$ a finite set. Then,
\[  \Rank(\Aut(A^G)) \leq  \Rank(\End(A^G)) . \]
In particular, if $\End(A^G)$ is finitely generated, then $\Aut(A^G)$ is finitely generated. 
\end{corollary}

In order to continue our study on the ranks of $\End(A^G)$, the next lemma is a key ingredient.

\begin{lemma}[Proposition 1.6.2 in \cite{CSC10}]\label{le-epi}
Let $A$ be a set, $G$ a group, and $N$ a normal subgroup of $G$. There exists a monoid epimorphism $\Phi : \End(A^G) \to \End(A^{G/N})$.
\end{lemma}

\begin{corollary}\label{le:rankCA}
With the notation of Lemma \ref{le-epi},
\begin{enumerate}
\item $\Rank(\End(A^{G/N})) \leq \Rank( \End(A^G) )$.
\item  $\Rank(\End(A^{G/N}) : \Aut(A^{G/N}) ) \leq \Rank( \End(A^G) : \Aut(A^G))$.
\end{enumerate}
\end{corollary}
\begin{proof}
The result follows by Lemmas \ref{le-epi-mon} and \ref{le-epi}.
\end{proof}

Using Corollary \ref{le:rankCA} (1.), it was shown in \cite[Theorem 7]{C20} that if $G$ is a group with an infinite descending chain of normal subgroups of finite index, then $\End(A^G)$ is not finitely generated. We shall prove the analogous result of the previous theorem for the relative rank of $\Aut(A^G)$ in $\End(A^G)$. 

As we shall see, in order to prove this, we need a lower bound for the relative rank of $\Aut(A^G)$ in $\End(A^G)$ when $G$ is a finite group. In order to describe it, for any subgroup $H$ of $G$, let $[H] := \{ gHg^{-1} : g \in G\}$ be the conjugacy class of $H$. Given two conjugacy classes $[H]$ and $[K]$, write $[H] \leq [K]$ if there exists $g \in G$ such that $H \subseteq gKg^{-1}$; as $G$ is finite, this is a well-defined partial order on the conjugacy classes of subgroups of $G$. Consider the set of edges of the digraph associated to this partial order:
\[ \mathcal{E}_{G} := \{ ([H], [K]) : H,K \leq G, \ [H] \leq [K] \}. \] 
Let $n(G)$ be the number of normal subgroups of $G$ and $r(G)$ the number of conjugacy classes of subgroups of $G$. As normal subgroups of $G$ are in one-to-one correspondence with conjugacy classes of size $1$, we have $n(G) \leq r(G)$. Finally, let $I_2(G)$ be the set of subgroups of $G$ of index $2$. 

\begin{theorem}[Theorem 7 in \cite{CRG19}] \label{lower-bound}
For any finite set $A$ with at least two elements and any finite group $G$, 
\[ \Rank( \End(A^G) : \Aut(A^G)) \geq \begin{cases} 
\vert \mathcal{E}_G \vert - \vert I_2(G) \vert & \text{ if } \vert A \vert =2, \\
\vert \mathcal{E}_G \vert & \text{ otherwise,}
\end{cases} \]
with equality if and only if $G$ is a Dedekind group (i.e., all subgroups of $G$ are normal).
\end{theorem}

Now we may prove the analogous result of \cite[Theorem 7]{C20} for the relative rank of $\Aut(A^G)$ in $\End(A^G)$.

\begin{theorem}\label{relative}
Let $A$ be a finite set with at least two elements, and let $G$ be a group such that there exists an infinite decreasing chain
\[ G > N_1 > N_2 > N_3 > \dots > N_k > \dots \]
where, for all $i \geq 1$, $N_i$ is a normal subgroup of $G$ of finite index. Then, the relative rank of $\Aut(A^G)$ in $\End(A^G)$ is infinite. 
\end{theorem}
\begin{proof}
As $G/N_i$ is a finite group, by Theorem \ref{lower-bound} we have
\[ \Rank(\End(A^{G/N_i}) : \Aut(A^{G/N_i}) )\geq  \vert \mathcal{E}_{G/N_i} \vert - \vert I_2(G/N_i) \vert. \]
Observe that
\[ \vert \mathcal{E}_{G/N_i} \vert \geq r(G/N_i) \geq n(G/N_i). \] 
By the Correspondence Theorem, the normal subgroups of $G/N_i$ are in one-to-one correspondence with normal subgroups $N$ of $G$ such that $N_i \leq N \leq G$. Thus, $n(G/N_i)$ is at least $i+1$, as $N_i < N_{i-1} < \dots < N_1 < G$ are intermediate normal subgroups. As at most one of these intermediate subgroups has index $2$, we have
\[  \Rank(\End(A^{G/N_i}) : \Aut(A^{G/N_i}) ) \geq n(G/N_i) -  \vert I_2(G/N_i) \vert \geq (i+1) - 1 =  i, \quad \text{ for all } i \geq 1. \]
By Corollary \ref{le:rankCA}, for all $i \geq 1$, we have
\[  \Rank( \End(A^G) : \Aut(A^G)) \geq  \Rank(\End(A^{G/N_i}) : \Aut(A^{G/N_i}) )  \geq i.   \]
This shows that $\Rank( \End(A^G) : \Aut(A^G))$ cannot be finite. 
\end{proof}

\begin{remark}
The hypothesis of Theorem \ref{relative} may seem unnatural, but it is in fact satisfied by many infinite groups, such as every infinite residually finite group. First of all, it is easy to check that a group $G$ is residually finite if and only if for every non-identity $g \in G$ there is a normal subgroup $N$ of finite index in $G$ such that $g \not \in N$. Hence, when $G$ is infinite residually finite, we shall construct an infinite descending chain $G = N_0 > N_1 > N_2 > \dots$ of normal subgroups of finite index. Suppose that the group $N_i$ has been constructed. Clearly, $N_i$ is non-trivial as it is a subgroup of finite index in an infinite group. Take $g \in N_i < G$ with $g \neq e$. Then, there exists a normal subgroup $N_i^\prime$ of finite index in $G$ such that $g \not \in N_i^\prime$. Define $N_{i+1} := N_i \cap N_i^\prime$. Note that $N_{i+1}$ is properly contained in $N_i$, as, otherwise, $N_i = N_i \cap N_i^\prime$ implies that $N_i \subseteq N_i^\prime$, contradicting the existence of $g$. Since intersections of normal subgroups are normal subgroups, and intersections of subgroups of finite index have finite index, then $N_{i+1}$ has the required properties.  
\end{remark}

%%%%%%%%%%%%%%%%%%%%%%%%%%%%%%%%%%%%%%%%%%%%%%%%%%%%
\section{Generating monoids of linear cellular automata} \label{LCA}

Let $V$ a vector space over a field $\mathbb{F}$. For any group $G$, the configuration space $V^G$ is also a vector space over $\mathbb{F}$ equipped with the pointwise addition and pointwise scalar multiplication. Let $\End_{\mathbb{F}}(V^G)$ be the set of all cellular automata of $V^G$ that are also $\mathbb{F}$-linear endomorphisms of $V^G$. Note that $\End_{\mathbb{F}}(V^G)$ is not only a monoid with respect to composition, but also an $\mathbb{F}$-algebra (i.e. a vector space over $\mathbb{F}$ equipped with a bilinear binary product), because, again, we may equip $\End_{\mathbb{F}}(V^G)$ with the pointwise addition and pointwise scalar multiplication. In particular, $\End_{\mathbb{F}}(V^G)$ is also a ring. However, in this section, we shall focus on the monoid structure of $\End_{\mathbb{F}}(V^G)$ with respect to composition; hence, the rank of $\End_{\mathbb{F}}(V^G)$ is the smallest cardinality of a subset $T$ that generates $\End_{\mathbb{F}}(V^G)$ using only composition. 

\begin{remark}
Given a cellular automaton $\tau : V^G \to V^G$, it is straightforward to show that $\tau$ is linear if and only if its local function $\mu : V^S \to V$ is linear. If $V$ is not finite-dimensional, then there are uncountably many linear functions $\mu : V^S \to V$. It follows that $\End_{\mathbb{F}}(V^G)$ is uncountable, so its rank is $\vert \End_{\mathbb{F}}(V^G) \vert$.
\end{remark}

To address the case when $V$ is a finite-dimensional vector space, we shall introduce an equivalent way to describe the monoid $\End_{\mathbb{F}}(V^G)$. 

For any ring $R$, the \emph{group ring} $R[G]$ consists of the set of all functions $f : G \to R$ of finite support equipped with pointwise addition and scalar multiplication. Equivalently, we may see the elements of $R[G]$ as formal finite sums $\sum_{g \in G} a_g g$ with $a_g \in R$; then, multiplication in $R[G]$ is defined naturally using the multiplications of $G$ and $R$:
\[ \sum_{g \in G} a_g g \sum_{h \in G} b_h h = \sum_{g, h \in G} a_g b_h  gh. \]
By \cite[Theorem 8.5.2]{CSC10}), for any group $G$ and any vector space $V$ over $\mathbb{F}$, we have the following isomorphism of $\mathbb{F}$-algebras: 
\[ \End_{\mathbb{F}}(V^G) \cong \End(V)[G]. \]
Denote by $\Aut_{\mathbb{F}}(V^G)$ the group of units of $\End_{\mathbb{F}}(V^G)$. 

\begin{theorem}\label{infinite-field}
Let $G$ be any group, and let $V$ be a finite-dimensional vector space over an infinite field $\mathbb{F}$. Then, neither $\End_{\mathbb{F}}(V^G)$ nor $\Aut_{\mathbb{F}}(V^G)$ are finitely generated.
\end{theorem}
\begin{proof}
Consider the agumentation homomorphism $\epsilon : \End(V)[G] \to \End(V)$ defined by 
\[ \epsilon \left( \sum_{g \in G} a_g g \right) := \sum_{g \in G} a_g, \quad \text{ for every }   \sum_{g \in G} a_g g \in \End(V)[G], \] 
see \cite[Definition 3.2.9]{MS02}. As this is an epimorphism, Lemma \ref{le-epi-mon} implies that $\Rank \left( \End(V)[G] \right)$ is at least $\Rank(\End(V))$. But now we have an epimorphism from $\End(V)$ to $\mathbb{F}$ given by the determinant map. Hence, $\Rank(\mathbb{F}) \leq \Rank(\End(V))$. As the multiplicative structure of an infinite field is not finitely generated (see \cite[p. 46]{H2000}), we conclude that $\End(V)[G] \cong \End_{\mathbb{F}}(V^G)$ is not finitely generated.

To show that $\Aut_{\mathbb{F}}(V^G)$ is not finitely generated, let $U$ be the group of units of $\End(V)[G]$. Consider the restriction of the agumentation homomorphism $\epsilon\vert_U : U \to \Aut(V)$. This is an epimorphism, as a preimage of any $a \in \Aut(V)$ is $ag \in U$, for any $g \in G$. Hence, $\Rank(\Aut(V)) \leq \Rank(U)$. Now again we have an epimorphsim from $\Aut(V)$ to $\mathbb{F}^\star$ given by the determinant map, so $\Rank(\mathbb{F}^\star) \leq \Rank(\Aut(V))$. As $\mathbb{F}^\star$ is not finitely generated, then $U \cong \Aut_{\mathbb{F}}(V^G)$ is not finitely generated. 
\end{proof}

We shall develop some machinery in order to address the case when the field $\mathbb{F}$ is finite. 

A group $G$ is called \emph{L-surjunctive} if for any finite-dimensional vector space $V$, every injective linear cellular automaton $\tau : V^G \to V^G$ is surjective. It is known that all sofic groups are L-surjunctive (see \cite[Theorem 1.2]{CSC07}), so in particular all resudually finite groups and all amenable groups are L-surjunctive. An analogous result to Theorem \ref{th-surjunctive} establishes that $G$ is L-surjunctive if and only if $\End_{\mathbb{F}}(V^G)$ is directly finite (see \cite[Theorem 1.3]{CSC07}). Hence, we have the following lemma.

\begin{lemma}\label{formula-LCA}
Let $G$ be an L-surjunctive group, and let $V$ be a finite-dimensional vector space. Then,
\[ \Rank(\End_{\mathbb{F}}(V^G)) = \Rank(\Aut_{\mathbb{F}}(V^G)) + \Rank(\End_{\mathbb{F}}(V^G) : \Aut_{\mathbb{F}}(V^G)). \]
\end{lemma}

\begin{corollary}\label{Rank-Lin}
Let $G$ be an L-surjunctive group and $V$ a finite-dimensional vector space. Then,
\[ \Rank(\Aut_{\mathbb{F}}(V^G)) \leq \Rank(\End_{\mathbb{F}}(V^G)). \]
\end{corollary}

The next lemma is the linear analogue of Lemma \ref{le-epi}.

\begin{lemma}\label{le-epi-lin}
Let $V$ be a vector space, $G$ a group, and $N$ a normal subgroup of $G$. There exists an epimorphism $\Phi : \End_{\mathbb{F}}(V^G) \to \End_{\mathbb{F}}(V^{G/N})$. 
\end{lemma}
\begin{proof}
Consider the agumentation ideal $\Delta(G,N)$ of $\End(V)[G]$; this is, the left ideal of $\End(V)[G]$ generated by the set $\{ a - \id  : a \in N  \}$. By \cite[Corollary 3.3.5]{MS02}, $\Delta(G,N)$ is a two-sided ideal and 
\[ \frac{\End(V)[G]}{\Delta(G,N)} \cong \End(V)[G/N].\]
The result follows, as $\End(V)[G] \cong \End_{\mathbb{F}}(V^G)$ and $\End(V)[G/N] \cong \End_{\mathbb{F}}(V^{G/N})$. 
\end{proof}

\begin{corollary}\label{Cor-Lin}
Let $V$ be a vector space, $G$ a group, and $N$ a normal subgroup of $G$. Then,
\[ \Rank(\End_{\mathbb{F}}(V^{G/N})) \leq \Rank( \End_{\mathbb{F}}(V^G)). \]
\end{corollary}
\begin{proof}
The result follows by Lemmas \ref{le-epi-mon} and \ref{le-epi-lin}. 
\end{proof}

Recall that $G$ is an \emph{indicable group} if there exists an epimorphism from $G$ to $\mathbb{Z}$ (see \cite[Definition 3]{B93}).

\begin{theorem}\label{LCA}
Let $G$ be an indicable group, and let $V$ be a finite-dimensional vector space over a finite field $\mathbb{F}$. Then, the monoid $ \End_{\mathbb{F}}(V^{G})$ is not finitely generated.
\end{theorem}
\begin{proof}
First we prove the theorem for $V = \mathbb{F}$. Suppose first that $p =: \text{char}(\mathbb{F})$ is not $2$. As $G$ is indicable, we may find a normal subgroup $N_k$ such that $G/N_k \cong \mathbb{Z}_{2^k}$, for every $k \in \mathbb{N}$. As $p \nmid 2^k$, it is a consequence of Perlis-Walker theorem (see \cite[Theorem 3.5.4]{MS02}) that $ \End_{\mathbb{F}}(\mathbb{F}^{\mathbb{Z}_{2^k}}) \cong  \mathbb{F}[\mathbb{Z}_{2^k}]$ decomposes as a direct sum $K_1 \oplus K_2 \oplus \dots \oplus K_{t}$, where each $K_i$ is a simple field extension of $\mathbb{F}$ by a primitive root of unity, and where $t$ is at least $k+1$, which is the number of divisors of $\vert \mathbb{Z}_{2^k} \vert = 2^k$. Therefore,
\[ \Aut_{\mathbb{F}}(\mathbb{F}^{\mathbb{Z}_{2^k}})\cong K_1^\star \oplus K_2^\star \oplus \dots \oplus K_{t}^\star.  \]  
As groups with multiplication, we have $K_i^\star \cong \mathbb{Z}_{p^{r_i} - 1}$, where $p^{r_i} = \vert K_i \vert$. Considering the fact that $\mathbb{Z}_a \oplus \mathbb{Z}_b \cong \mathbb{Z}_{ab}$ if and only if $\gcd(a,b)=1$, and since $p - 1\mid p^{r_i} - 1$, for all $i$, we conclude that 
\[ \Rank(  \Aut_{\mathbb{F}}(\mathbb{F}^{\mathbb{Z}_{2^k}})  ) =  \Rank( \mathbb{Z}_{p^{r_1} - 1} \oplus \dots \oplus \mathbb{Z}_{p^{r_t} - 1}) = t \geq k+1.  \] 
By Corollaries \ref{Rank-Lin} and \ref{Cor-Lin},
\[  \Rank( \End_{\mathbb{F}}(\mathbb{F}^{G} ) ) \geq \Rank( \End_{\mathbb{F}}(\mathbb{F}^{\mathbb{Z}_{2^k}} )) \geq \Rank( \Aut_{\mathbb{F}}(\mathbb{F}^{\mathbb{Z}_{2^k}} )) \geq k +1 \]
for all $k \in \mathbb{N}$. Therefore, $\End_{\mathbb{F}}(\mathbb{F}^{G})$ is not finitely generated.

For the case when $\text{char}(\mathbb{F})=2$, we proceed with an analogous argument as in the previous paragraph, except that the normal subgroups $N_k$ of $G$ are chosen such that $G/N_k \cong \mathbb{Z}_{3^k}$, for every $k \in \mathbb{N}$.  

Now let $V$ be a $d$-dimensional vector space over $\mathbb{F}$. As there exists an epimorphism from $G$ to $\mathbb{Z}$, we also have a ring epimorphism from $ \End_{\mathbb{F}}(V^{G}) \cong \End(V)[G]$ to $\End(V)[\mathbb{Z}]$ (see \cite[Corollary 3.2.8]{MS02}). By \cite[Corollary 8.7.8.]{CSC10},  $\End(V)[\mathbb{Z}]$ is isomorphic to the algebra of matrices $\text{Mat}_{d \times d}(\mathbb{F}[\mathbb{Z}])$. As the ring $\mathbb{F}[\mathbb{Z}]$ is commutative, the determinant map is an epimorphism from $\text{Mat}_{d \times d}(\mathbb{F}[\mathbb{Z}])$ to $\mathbb{F}[\mathbb{Z}]$. Hence, by Lemma \ref{le-epi-mon}, $\Rank(\mathbb{F}[\mathbb{Z}]) \leq \Rank(\End(V)[G])$, and the result follows as $\mathbb{F}[\mathbb{Z}]$ is not finitely generated by the argument of the previous paragraph.  
\end{proof}

An invertible element in $\mathbb{F}[G]$ is a \emph{trivial unit} if it has the form $ag \in \mathbb{F}[G]$ with $a \in F^\star$, $g \in G$. The well-known unit conjecture states that if $G$ is a torsion-free group, then $\mathbb{F}[G]$ does not contain non-trivial units; recently, a counter example for this was found in \cite{G21} when $G$ is virtually abelian. However, many groups do satisfy Kaplansky's unit conjecture, such as torsion-free nilpotent groups (\cite[Corollary 8.5.5]{MS02}) and locally indicable groups (\cite{H40}).   

\begin{proposition}\label{last}
Let $G$ be a finitely generated indicable group and let $\mathbb{F}$ be a field. Then, $\Aut_{\mathbb{F}}(\mathbb{F}^{G})$ is finitely generated if and only if $\mathbb{F}$ is a finite field. 
\end{proposition}
\begin{proof}
As $G$ satisfies the unit conjecture, then it is clear that $\Aut_{\mathbb{F}}(\mathbb{F}^{G}) \cong G \oplus \mathbb{F}^\star$ and
\[ \Rank(\Aut_{\mathbb{F}}(\mathbb{F}^{G})) \leq\Rank(G) + \Rank(\mathbb{F}^\star).  \]
If $\mathbb{F}$ is a finite field, then $\Rank(\mathbb{F}^\star) = 1$, as $\mathbb{F}^\star$ is cyclic. Hence, $\Rank(\Aut_{\mathbb{F}}(\mathbb{F}^{G})) \leq \Rank(G) + 1$, which shows that $\Aut_{\mathbb{F}}(\mathbb{F}^{G})$ is finitely generated. Conversely, if $\mathbb{F}$ is an infinite field, then $\mathbb{F}^\star$ is not finitely generated (see \cite[p. 46]{H2000}). Since $\Aut_{\mathbb{F}}(\mathbb{F}^{G}) \cong G \oplus \mathbb{F}^\star$ must contain a generating set for $\mathbb{F}^\star$, it follows that $\Aut_{\mathbb{F}}(\mathbb{F}^{G})$ is not finitely generated.
\end{proof}

\begin{remark}
Theorem \ref{LCA} and Proposition \ref{last} show that, for any finitely generated indicable group and finite field $\mathbb{F}$,  $\End_{\mathbb{F}}(\mathbb{F}^{G})$ is not finitely generated but $\Aut_{\mathbb{F}}(\mathbb{F}^{G})$ is finitely generated.  By Lemma \ref{formula-LCA}, we must have that the relative rank of $\Aut_{\mathbb{F}}(\mathbb{F}^{G})$ in $\End_{\mathbb{F}}(\mathbb{F}^{G})$ is infinite. In particular, this is true for $G=\mathbb{Z}^d$, with $d \in \mathbb{N}$, an important case in symbolic dynamics and the theory of cellular automata.  
\end{remark}

\section*{Acknowledgments}

This work was supported by a CONACYTg from the Government of Mexico. The author thanks James East for his kind orientation on ranks of monoids.

%%%%%%%%%%%%%%%%%%%%%%%%%%%%%%%%%%%%%%%%%%%%%%%%%%%%%%%%%%%%%%%%%%%%%%%%%%%%%%%%%%%%%%%%%%%%%

\end{document}